\documentclass[12pt,reqno]{article}

\usepackage{amsmath,amsthm,amsfonts,amssymb,mathrsfs,url,tikz}
\usepackage[square,comma,sort&compress,numbers]{natbib}
\usepackage{float}
\usepackage{algorithm}
\usepackage{algpseudocode}

\newtheorem{thm}{Theorem}[section]
\newtheorem{lem}[thm]{Lemma}
\newtheorem{cor}[thm]{Corollary}
\newtheorem{prop}[thm]{Proposition}
\newtheorem{prob}[thm]{Problem}

\theoremstyle{definition}

\newtheorem{exm}[thm]{Example}

\numberwithin{equation}{section}

\newcommand{\cB}{\mathcal B}
\newcommand{\cP}{\mathcal P}

\newcommand{\ZZ}{\mathbb Z}

\title{Engel's Interval Packing Problem in the Boolean Lattice}
\author{Yuxian Dong, Jianxi Mao
\thanks{
	Corresponding author.
    \newline\hspace*{3mm}
    {\it Email addresses:}\quad
    yuxiand@hotmail.com (Yuxian Dong); maojx@dlut.edu.cn (Jianxi Mao)
    }
}
\date{\footnotesize
 School of Mathematical Sciences, Dalian University of Technology, Dalian 116024, P.R. China\\
}
\begin{document}

\maketitle

\begin{abstract}
Let \(\cB_n\) be the Boolean lattice of all subsets of \([n]\) and let \(\cP_{n;\ell,u}\) be the subposet of \(\cB_n\) induced by the consecutive levels \(\ell,\ell+1,\ldots,u\). 
We determine $\nu_{n;\ell,u}$, 
the maximum size of a family of pairwise disjoint maximal intervals in 
$\mathcal P_{n;\ell,u}$, whenever
 \(u\le ({n+\ell^2})/({\ell+1})\).
This completely settles Engel's problem~[Combin. Probab. Comput., 1996]. 
The proof is constructive. 
We also record consequences for weakly cross-intersecting set-pair systems and discuss the three-level case.
\end{abstract}


\\
\noindent
{\bf Note added after completion:}
After completing the manuscript, we became aware that the main result of this paper is already contained in earlier work on the Stanley depth conjecture. In particular, the results of Keller, Shen, Streib and Young \cite{KSSY11} and Ge, Lin and Shen \cite{GLS12} directly imply our main result.

\section{Introduction}
Let \([n]=\{1,2,\dots,n\}\). For $0\le \ell \le u\le n$, define
$$
\cP_{n;\ell,u}
=
\{X\subseteq [n]: \ell\le |X|\le u\},
$$
ordered by inclusion.
Equivalently, \(\cP_{n;\ell,u}\) is the subposet of the Boolean
lattice $\mathcal B_n$ induced by the consecutive levels \(\ell,\ell+1,\ldots,u\).
An interval in \(\mathcal P_{n;\ell,u}\) is \emph{maximal} if and only if it is of the form
\[
[A,B]
=
\{X\subseteq [n]: A\subseteq X\subseteq B\},
\qquad |A|=\ell,\ |B|=u.
\]
We denote by \(\nu_{n;\ell,u}\) the maximum number of pairwise disjoint maximal intervals in \(\cP_{n;\ell,u}\).

It is clear that \(\nu_{n;\ell,u}\le \binom n\ell,\)
since two disjoint maximal intervals cannot have the same lower endpoint.
Suppose this equality holds.
Then there are \(\binom n\ell\) pairwise disjoint maximal intervals in \(\cP_{n;\ell,u}\),
and each maximal interval contains exactly \(u-\ell\) sets of size \(\ell+1\).
Since the total number of sets of size
$\ell+1$ in $\mathcal B_n$ is $\binom{n}{\ell+1}$, we have
\[
(u-\ell)\binom n\ell
\le
\binom n{\ell+1}.
\]
It follows that \(u\le ({n+\ell^2})/({\ell+1}).\)
Note that under this bound, if \(u>\ell\),
then 
\[
n\ge (\ell+1)u-\ell^2=u+\ell(u-\ell)\ge u+\ell.
\]
This implies that \(\ell< u\le n-\ell\) and 
\(\binom n\ell\le \binom nu\).

A {\it packing} in a poset \(P\) is a family of pairwise disjoint subposets of \(P\).
Packing problems have attracted increasing attention in extremal set theory and poset theory~\cite{BE97,Eng96,Fur88,GLT19,Tom20,Tuz87}.
In 1996, Engel~\cite{Eng96} studied the interval packing problem in the Boolean lattice. 
He proved that $\nu_{n;\ell,u}=\binom n\ell$ 
when \(\ell=1,2\) and $u\le (n+\ell^2)/(\ell+1)$.
He then asked whether the same statement remains true for all \(\ell\ge 3\).

\begin{prob}(\cite[Problem 1]{Eng96})\label{engp}
	Is it true that for \(\ell\ge 3\) and $u\le (n+\ell^2)/(\ell+1)$, one has
	$$
	\nu_{n;\ell,u}=\binom n\ell?
	$$
\end{prob}

In this paper, we completely settle Problem~\ref{engp}. 
Writing \(u=\ell+r\), the above condition is equivalent to \(n\ge (\ell+1)\,r+\ell\). 
Our main result is the following.

\begin{thm}\label{thm:main}
	Let $\ell,r\ge 0$. 
	If \(n\ge (\ell+1)\,r+\ell\), then
	\[
	\nu_{n;\ell,\ell+r}=\binom n\ell.
	\]
\end{thm}

The paper is organized as follows. 
Section~\ref{sec:critical} proves the critical case $n=(\ell+1)\,r+\ell$,
which plays a crucial role in the proof of the full result. 
First, we construct an \(r\)-set $C_T$ for each \(\ell\)-set $T\subseteq [n].$
Then we prove that the intervals
\[
[T,T\cup C_T] \qquad \textrm{for all } \, T\in \binom{[n]}{\ell}
\]
form the desired interval packing in $\cP_{n;\ell,\ell+r}$.
In Section~\ref{sec:main-proof}, we prove Theorem~\ref{thm:main} by induction. 
Section~\ref{sec:further-remarks} contains consequences and related problems.
The Appendix records the construction in algorithmic form.

\section{The critical case $n=(\ell+1)\,r+\ell$}\label{sec:critical}

Given a sequence
$
(a_1,\ldots,a_N)
$,
arrange it in clockwise order around a circle, 
so that $a_1$ comes after $a_N$. 
We regard a sequence as a cyclic sequence whenever its indices are read modulo \(N\); that is,
\[
a_{i+N}=a_i,
\qquad\text{for all } i\in\ZZ.
\]

We first recall a cycle lemma, which plays a key role in our proof. 
The cycle lemma is a powerful tool and has many applications in combinatorics~\cite{DZ90, DM47,STW22}.
Here we give a negative-sum version,
which is a simple variant of~\cite[Theorem 2.1]{Ran60}.
For the sake of completeness, we present a proof.

\begin{lem}[Cycle lemma]
\label{lem:cycle}
Let \(\left(a_1,a_2,\dots,a_{N}\right)\) be an integer sequence with $a_i\le 1$ for all $1\le i\le N$, and
$$
\sum_{i=1}^{N}a_i=-1.
$$
Then there is a unique index \(v\in[N]\) such that for all $0\le j\le N-1$,
$$
\sum_{i=0}^{j} a_{v+i}<0.
$$
\end{lem}

\begin{proof}
Let
\[
s_0=0,\qquad \textrm{and}\qquad s_j=\sum_{i=1}^{j}a_i,\quad 1\le j\le N .
\]
Then \(s_N=-1<s_0\). 
Choose \(v-1\in\{0,1,\ldots,N-1\}\) to be the last index at which the maximum of $s_0,s_1,\ldots,s_N$ is attained, that is,
\[
s_{v-1}=\max_{0\le j\le N}s_j.
\]
Let \(0\le k\le N-1\). 
If \(v+k\le N\), then
\[
\sum_{i=0}^{k}a_{v+i}
=
s_{v+k}-s_{v-1}
<0,
\]
since \(v-1\) is the last occurrence of the maximum of $(s_0,\ldots,s_N)$. 
If \(v+k> N\), then
\[
\sum_{i=0}^{k}a_{v+i}
=
(s_N-s_{v-1})+s_{v+k-N}
=
-1-\left(\max_{0\le j\le N}s_j-s_{v+k-N}\right)
< 0.
\]
Thus \(v\) has the desired property.

Suppose that \(v'\) also has the desired property and $v'\ne v$. 
Then
\begin{equation*}
s_N
=
\sum_{i=v'}^{v-1} a_i+\sum_{i=v}^{v'-1} a_i\le (-1)+(-1)=-2.
\end{equation*}
This contradicts $s_N=-1$. Hence, the index $v$ is unique.
\end{proof}

Throughout this section, we assume that \(n=(\ell+1)\,r+\ell\) and \(r\ge 1\). 
For a subset $T\subseteq [n]$, 
define a weight function on $[n]$ by
\begin{equation}\label{weightdef}
w_T(i)=
\begin{cases}
	-r, & \qquad \textrm{if }\,\, i\in T,\\
	1, & \qquad \textrm{if }\,\, i\notin T.
\end{cases}
\end{equation}
We call $(w_T(1),\ldots,w_T(n))$ the {\it weight sequence} of $T$.
Define its {\it partial sums} by
$$
s_T(0)=0,\qquad \textrm{and}\qquad s_T(i)=\sum_{j=1}^i w_T(j),\qquad \textrm{for } \, i=1,\dots,n.
$$
Since $n=(\ell+1)\,r+\ell$,
if $T\in \binom{[n]}{\ell},$ then
\begin{equation}\label{rr}
	s_T(n)=\sum_{i=1}^n w_T(i)=(n-\ell)-\ell\,r=r.
\end{equation}
Similarly, if \(S\in \binom{[n]}{\ell+1}\), then
\begin{equation*}
	s_S(n)=\sum_{i=1}^n w_S(i)=(n-\ell-1)-(\ell+1)\,r=-1.
\end{equation*}

\begin{prop}
For any \(T\in\binom{[n]}{\ell}\) and any \(S\in\binom{[n]}{\ell+1}\), the total sums of their weight sequences are
\[
s_T(n)=r
\qquad\text{and}\qquad
s_S(n)=-1.
\]
\end{prop}

Consider the weight sequence $(w_S(1),\ldots,w_S(n))$ for an $(\ell+1)$-subset $S\subseteq [n]$.
Since all weights are integers and at most $1$,
by Lemma~\ref{lem:cycle},
there is a unique index \(v\in [n]\) such that
\begin{equation}\label{weightleq}
\sum_{i=0}^{j} w_S(v+i)\le -1, \qquad \textrm{for all }\,\, 0\le j\le n-1.
\end{equation}
Taking $j=0$, we have $w_S(v)<0$. Hence $v\in S$.
We call the unique index $v$ the {\it nice position} of $S$. 

For an $\ell$-subset $T\subseteq [n]$, define
$$
C_T
=
\{v\in [n]\setminus T: v \, \textrm{ is the nice position of } \, T\cup\{v\}\}.
$$
We now give an explicit characterization of the set \(C_T\).
Let 
\begin{equation}\label{T}
	M=\max(s_T(0),s_T(1),s_T(2),\ldots, s_T(n)=r)
\end{equation}
be the maximum of the partial sum sequence of $T\in\binom{[n]}{\ell}$.
Since each upward step of the partial-sum sequence has size $1$, every integer between $1$ and $M$ occurs as a partial sum.
For $1\le m\le M$, let $i_m$ denote the first index such that $s_T(i_m)=m$.

\begin{lem}
\label{lem:children}
Assume that \(r\ge 1\).
For every \(T\in\binom{[n]}{\ell}\), we have \(|C_T|=r\).
Moreover,
\begin{equation}\label{first}
	C_T=\{i_{M-r+1},i_{M-r+2},\ldots,i_M\},
\end{equation}
where $M$ is the maximum of the partial sum sequence of $T$ defined in~\eqref{T}
and 
the index $i_m$ is the first index such that $s_T(i_m)=m$.
\end{lem}

\begin{proof} 
Let \(T\in\binom{[n]}{\ell}\) and $v\in [n]\setminus T$. 
Set 
$$
S=T\cup\{v\}\in\binom{[n]}{\ell+1}.
$$ 
We divide the proof into three steps.

\textbf{Step 1}:
We claim that \(v\in C_T\) if and only if $v\in [n]\setminus T$ and for every \(0\le j\le n-1\),
\begin{equation}\label{keyle}
	\sum_{i=0}^{j} w_T(v+i)\le r.
\end{equation}

Suppose that $v\in C_T.$
Then the weights \(w_T\) and \(w_S\) differ only at $v$, where \(w_T(v)=1\) and \(w_S(v)=-r\).
Thus for every \(0\le j\le n-1\),
\begin{equation*}
	\sum_{i=0}^{j} w_T(v+i)=1+\sum_{i=1}^{j} w_T(v+i)=1+\left(\sum_{i=0}^{j} w_S(v+i)-w_S(v)\right)\le r,
\end{equation*}
where the last inequality follows from~\eqref{weightleq}.

Conversely, suppose that $v\in [n]\setminus T$ and $v\notin C_T.$
By Lemma~\ref{lem:cycle},
there exists an index $0\le j\le n-1$ such that 
$
\sum_{i=0}^{j}w_S(v+i)\ge 0.
$
Thus
\begin{equation*}
	\sum_{i=0}^{j} w_T(v+i)=1+\sum_{i=1}^{j} w_T(v+i)=1+\left(\sum_{i=0}^{j} w_S(v+i)-(-r)\right)\ge r+1.
\end{equation*}
This proves the claim.

\textbf{Step 2}: We claim that \(v\in C_T\) if and only if
\begin{equation}\label{chara}
	s_T(v-1)=\max_{0\le k\le v-1}s_T(k)
	\qquad\text{and}\qquad
	s_T(v-1)\ge M-r.
\end{equation}

Let $k\in \{1,2,\ldots, n\}$ and $k \equiv v+j\pmod{n}$.
We distinguish two cases, according to whether $k<v$ or $k\ge v$.
The inequality~\eqref{keyle} yields
\begin{align*}
&	\sum_{i=0}^{j} w_T(v+i)=s_T(n)-s_T(v-1)+s_T(k)\le r, \qquad \textrm{if }\,\,k<v\quad (\,\textrm{i.e., } \, k=v+j>n);\\
&	\sum_{i=0}^{j} w_T(v+i)=s_T(k)-s_T({v-1})\le r, \qquad \textrm{if }\,\,k\ge v \quad (\,\textrm{i.e., } \, k=v+j\le n).
\end{align*}
(Note that $j=n-1$ implies that $k=v-1$.)
Since $s_T(n)=r$,
the condition~\eqref{chara} follows. 

Conversely,
suppose that there exists an index $j$ such that~\eqref{keyle} does not hold.
If $k<v,$
then $s_T({v-1})<s_T(k)$.
If $k\ge v$,
then $s_T({v-1})<s_T(k)-r$.
Hence, the claim holds. 

\textbf{Step 3}: Construction of $C_T$.

Clearly, \(w_T({i_m})=1\) and \(s_T(i_m-1)=m-1\). 
Thus \(i_m\notin T\), 
and \(s_T(i_m-1)\) is a prefix maximum, i.e., 
$$s_T(i_m-1)=\max_{0\le j\le i_m-1}s_T(j).$$
Then by~\eqref{chara},
$
i_m\in C_T
$
if 
$s_T(i_m-1)=m-1\ge M-r,$
which implies that $m\ge M-r+1$.
Thus we have 
$$
\{i_{M-r+1},i_{M-r+2},\ldots,i_M\}\subseteq C_T.
$$
Conversely, let $v\in C_T$, 
and suppose \(s_T(v-1)=m-1\).
Since \(v\notin T\), we have \(w_T(v)=1\), and hence \(s_T(v)=m\).
By~\eqref{chara}, \(s_T(v-1)\) is a prefix maximum and \(m-1\ge M-r\), so \(v\) is the first index such that \(s_T(v)=m\).
Thus
$$
C_T=\{i_{M-r+1},i_{M-r+2},\ldots,i_M\}.
$$
In particular, \(|C_T|=r\).
This completes the proof.
\end{proof}

\begin{exm}
	Let \(\ell=3\), \(r=2\), and \(n=(\ell+1)r+\ell=11\). 
	For \(T\in \binom{[11]}{3}\), define
\begin{equation*}
	w_T(i)=
	\begin{cases}
		-2, & \qquad \textrm{if }\,\, i\in T,\\
		1, & \qquad \textrm{if }\,\, i\notin T.
	\end{cases}
\end{equation*}
	Let
	\[
	T=\{3,4,10\}\in\binom{[11]}{3}.
	\]
	The weight sequence is
	\[
	(w_T(1),\ldots,w_T(11))
	=
	(1,1,-2,-2,1,1,1,1,1,-2,1),
	\]
	and the partial sum sequence is
	\[
	(s_T(0),s_T(1),\ldots,s_T(11))
	=
	(0,1,2,0,-2,-1,0,1,2,3,1,2).
	\]
	Hence
    $
	M=\max (s_T(0),s_T(1),\ldots,s_T(11))=3.
	$
	By~\eqref{first}, 
    it suffices to find the first indices $i_m$ such that $s_T(i_m)=m$ for $2=M-r+1\le m\le M=3.$
    From Figure~\ref{ps} of the partial sums, $i_2=2$ and $i_3=9$.
    Hence,
    $
    C_T=\{2,9\}.
    $
	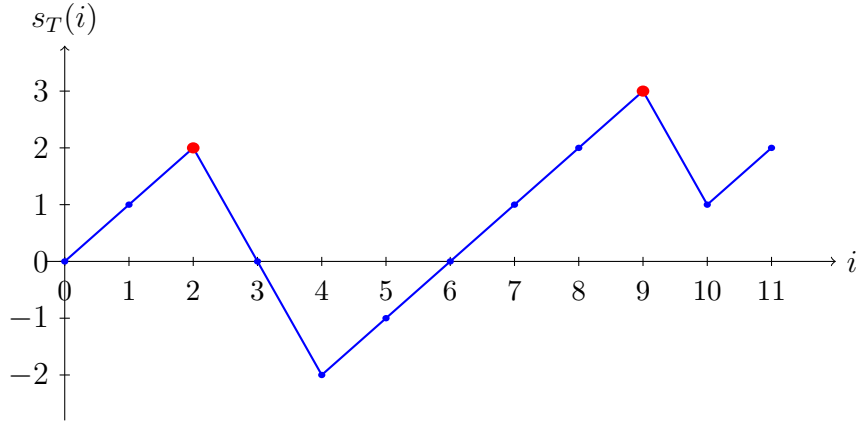
\begin{figure}[H]
		\centering
		\begin{tikzpicture}[xscale=0.85, yscale=0.75]
			
			\draw[->] (0,-2.8) -- (0,3.8) node[above] {\(s_T(i)\)};
			\draw[->] (-0.3,0) -- (12,0) node[right] {\(i\)};
			
			\foreach \y in {-2,-1,0,1,2,3}
			{
				\draw (-0.08,\y) -- (0.08,\y);
				\node[left] at (-0.08,\y) {\(\y\)};
			}
			
			\foreach \x/\lab in {
				0/\(0\),
				1/\(1\),
				2/\(2\),
				3/\(3\),
				4/\(4\),
				5/\(5\),
				6/\(6\),
				7/\(7\),
				8/\(8\),
				9/\(9\),
				10/\(10\),
				11/\(11\)
			}
			{
				\draw (\x,0.08) -- (\x,-0.08);
				\node[below, font=\small] at (\x,-0.15) {\lab};
			}
			
			\draw[thick, blue]
			(0,0) -- (1,1) -- (2,2) -- (3,0) -- (4,-2) -- (5,-1)
			-- (6,0) -- (7,1) -- (8,2) -- (9,3) -- (10,1) -- (11,2);
			
			\foreach \x/\y in {
				0/0,1/1,2/2,3/0,4/-2,5/-1,6/0,7/1,8/2,9/3,10/1,11/2
			}
			{
				\fill[blue] (\x,\y) circle (1.6pt);
			}
			
			\fill[red] (2,2) circle (2.8pt);
			\fill[red] (9,3) circle (2.8pt);

		\end{tikzpicture}
		\caption{The partial sums for \(T=\{3,4,10\}\).}\label{ps}
	\end{figure}
\end{exm}

Now we are ready to prove the critical case.
\begin{prop}
\label{prop:critical}
Let $\ell\ge 0$ and $r\ge 1$.
If \(n=(\ell+1)\,r+\ell\), then 
$$
\nu_{n;\ell,\ell+r}=\binom n\ell.
$$
\end{prop}

\begin{proof}
Let \(A_1,A_2\in\binom{[n]}{\ell}\) be two distinct subsets. 
Consider the intervals \([A_1,A_1\cup C_{A_1}]\) and \([A_2,A_2\cup C_{A_2}]\).
By Lemma~\ref{lem:children}, \(|C_{A_1}|=|C_{A_2}|=r\). 
Thus both are maximal intervals in \(\cP_{n;\ell,\ell+r}\).
Now we prove that any two constructed intervals are disjoint.
This will show that the intervals \([A,A\cup C_A]\) for each \(A\in\binom{[n]}{\ell}\) form a family of pairwise disjoint maximal intervals in \(\cP_{n;\ell,\ell+r}\).
Since there are $\binom{n}{\ell}$ such intervals and $\nu_{n;\ell,\ell+r}\le \binom{n}{\ell}$, the result follows.

Suppose that
\[
[A_1,A_1\cup C_{A_1}]\cap[A_2,A_2\cup C_{A_2}]\ne\varnothing.
\]
Then there exists a subset $H$ of $[n]$
such that 
$$
A_1\subseteq H \subseteq A_1\cup C_{A_1} \quad \textrm{and}\quad 
A_2\subseteq H \subseteq A_2\cup C_{A_2}.
$$
Let \(D=A_1\setminus A_2\) and \(E=A_2\setminus A_1\). 
Then \(|D|=|E|\ge 1\).
Since 
$$
D=A_1\setminus A_2 \subseteq H \subseteq A_2\cup C_{A_2}, \quad\textrm{and} \quad E=A_2\setminus A_1 \subseteq H \subseteq A_1\cup C_{A_1},
$$
we have \(D\subseteq C_{A_2}\) and \(E\subseteq C_{A_1}\). 

List the elements of $D\cup E$ as $(p_1,p_2,\ldots, p_m)$ in cyclic order, that is,
arranged in increasing order so that the successor of $p_m$ is $p_1$.
Choose \(d\in D\) whose successor lies in \(E\), 
and denote this successor by \(e\). 
Let 
$$I=[d,e)=\{d,d+1,\ldots,e-1\} \mod n,\quad J=[e,d)=\{e,e+1,\ldots,d-1\} \mod n$$
denote the two complementary half-open cyclic intervals.
Thus \(I\cup J=[n]\) and \(I\cap J=\varnothing\).

Since \(d\in D\subseteq C_{A_2}\) and \(e\in E\subseteq C_{A_1}\),
by~\eqref{keyle},
\begin{equation}\label{sum}
\sum_{i\in I}w_{A_2}(i)\le r\quad \textrm{and} \quad \sum_{i\in J}w_{A_1}(i)\le r.
\end{equation}
The interval \(I\) contains no point of \(D\cup E\) except \(d\). 
Hence \(w_{A_2}\) and \(w_{A_1}\) agree on \(I\setminus\{d\}\). 
Since $d\in D \subseteq A_1$ and $d\notin A_2$, 
we have \(w_{A_2}(d)=1\) and \(w_{A_1}(d)=-r\). 
Therefore
\[
\sum_{i\in I}w_{A_1}(i)=\sum_{i\in I} w_{A_2}(i)-(r+1).
\]
Since $|A_1|=\ell$, equation~\eqref{rr} gives \(\sum_{i\in [n]} w_{A_1}(i)=r\). 
Hence by~\eqref{sum}, 
\[
r\ge \sum_{i\in J}w_{A_1}(i)
=
r-\sum_{i\in I}w_{A_1}(i)
=
2r+1-\sum_{i\in I} w_{A_2}(i).
\]
This gives 
$$
\sum_{i\in I}w_{A_2}(i)\ge r+1,
$$ 
which contradicts~\eqref{sum}.
Hence, \[
[A_1,A_1\cup C_{A_1}]\cap[A_2,A_2\cup C_{A_2}]=\varnothing,
\]
and this completes the proof.
\end{proof}

\section{The general case}
\label{sec:main-proof}

We now prove~Theorem~\ref{thm:main}.
\begin{proof}[Proof of Theorem~\ref{thm:main}]	
We prove the theorem by induction on $n$, simultaneously for all admissible pairs $(\ell,r)$.
The assertion is immediate for $n=1,2$. 
Assume that the statement holds for $n-1$, 
that is, for every triple $(n-1,\ell, r)$ satisfying 
$
n-1\ge (\ell+1)\,r+\ell,
$
we have \(\nu_{n-1;\ell,\ell+r}=\binom{n-1}{\ell}\).
We now prove that it also holds for $n$.

Suppose that
$
n\ge (\ell+1)r+\ell.
$
The cases \(\ell=0\) and \(r=0\) are immediate. 
If \(\ell=0\), then one interval \([\varnothing,B]\) with \(|B|=r\) suffices;
if $r=0$, then all intervals are singletons $[A,A]$. 
Hence we may assume that $\ell,r\ge 1$.
If \(n=(\ell+1)\,r+\ell\), the result follows from Proposition~\ref{prop:critical}. 
Thus we may assume that \(n>(\ell+1)\,r+\ell\). 
Then \(n-1\ge (\ell+1)\,r+\ell\). 

We partition \(\cP_{n;\ell,\ell+r}\) into two subposets, according to whether the sets avoid $n$ or contain $n$. 
The subposet consisting of sets that avoid n is naturally isomorphic to \(\cP_{n-1;\ell,\ell+r}\). By the induction hypothesis applied to \((n-1,\ell,r)\), 
there are \(\binom{n-1}{\ell}\) pairwise disjoint maximal intervals in \(\cP_{n-1;\ell,\ell+r}\).
These intervals are also pairwise disjoint in \(\cP_{n;\ell,\ell+r}\).

The subposet consisting of sets that contain $n$ is naturally isomorphic, after deleting $n$, to \(\cP_{n-1;\ell-1,\ell+r-1}\).
Since
\[
n-1\ge \ell r+\ell-1,
\]
the induction hypothesis applied to $(n-1,\ell-1,r)$ gives
$
\binom{n-1}{\ell-1}
$
pairwise disjoint maximal intervals \([A',B']\) with \(|A'|=\ell-1\) and \(|B'|=\ell+r-1\). 
Adding \(n\) to every set gives the intervals
\[
\left[\{n\}\cup A',\{n\}\cup B'\right],
\]
which are maximal intervals in \(\cP_{n;\ell,\ell+r}\), 
and all sets in them contain \(n\).

The two families are mutually disjoint, since every set in the first family avoids $n$ and every set in the second family contains $n$. Therefore we obtain
\[
\binom{n-1}{\ell}+\binom{n-1}{\ell-1}
=
\binom n\ell
\]
pairwise disjoint maximal intervals in \(\cP_{n;\ell,\ell+r}\). 
Thus \(\nu_{n;\ell,\ell+r}\ge\binom n\ell\). 
Since level $\ell$ of \(\cB_n\) contains exactly $\binom{n}{\ell}$ subsets, 
we also have
\[
\nu_{n;\ell,\ell+r}\le \binom{n}{\ell}.
\]
Hence
\[
\nu_{n;\ell,\ell+r}=\binom{n}{\ell}.
\]
This completes the proof.
\end{proof}

\section{Consequences and related problems}
\label{sec:further-remarks}

\subsection{Weakly cross-intersecting set-pair systems}
Interval packings in the Boolean lattice are closely related to set-pair systems.
A family \(\{(A_i,B_i)\}_{i=1}^m\) of ordered pairs of subsets of \([n]\) is called a {\it Bollob\'as system} if 
\(A_i\cap B_i=\varnothing\) for all \(i\) and
\(A_i\cap B_j\ne\varnothing\) for all \(i\ne j\).
This notion goes back to Bollob\'as~\cite{Bol65}.
Related set-pair systems and intersection problems have been studied in~\cite{Tuz87,KNPV12,HF24,FW24}.

Tuza~\cite{Tuz87} studied set-pair systems satisfying the weaker condition that
\(A_i\cap B_i=\varnothing\)
for all \(i\), and for all \(i\ne j\),
\[
A_i\cap B_j\ne\varnothing
\qquad\text{or}\qquad
A_j\cap B_i\ne\varnothing .
\]
Following Kir\'aly, Nagy, P\'alv\"olgyi and Visontai~\cite{KNPV12}, we call such a family a {\it weakly cross-intersecting set-pair system}.
Clearly, every Bollob\'as system is weakly cross-intersecting.
Engel~\cite{Eng96} observed that interval packings in \(\cP_{n;p,n-q}\) can be encoded as weakly cross-intersecting set-pair systems.
Indeed,
a maximal interval \([X,Y]\) in \(\cP_{n;p,n-q}\) corresponds to the ordered pair
$(X,Y')$, where $Y'=[n]\setminus Y$.
Then
\[
|X|=p,\qquad |Y'|=q \qquad\text{and}\qquad X\cap Y'=\varnothing.
\]
Moreover, two maximal intervals \([X_1,Y_1]\) and \([X_2,Y_2]\) are disjoint if and only if
\[
X_1\cap Y'_2\ne\varnothing
\qquad\text{or}\qquad
X_2\cap Y'_1\ne\varnothing .
\]
Using weakly cross-intersecting set-pair systems,
Kir\'aly et al.~\cite{KNPV12} settled a question of Engel by proving that
$
\nu_{n;2,n-2}=10,
$
for all $n\ge 5$.

Let \(p,q\) be positive integers and let \(N=p+q+\left\lfloor\frac qp\right\rfloor\). 
Then \(N\ge p+(p+1)\left\lfloor\frac qp\right\rfloor\). 
Applying Theorem~\ref{thm:main} with \(\ell=p\), \(r=\left\lfloor\frac qp\right\rfloor\), and \(n=N\), we get
\[
\nu_{N;p,N-q}=\binom Np
=
\binom{p+q+\lfloor q/p\rfloor}{p}.
\]
For \(n\ge N\), this packing can be lifted from \([N]\) to \([n]\) by adding the elements of \([n]\setminus[N]\) to every upper endpoint. 
The following is immediate.

\begin{cor}\label{cor:set-pair}
	Let \(p,q\) be positive integers. If
	\(n\ge p+q+\left\lfloor\frac qp\right\rfloor,\)
	then
	\[
	\nu_{n;p,n-q}
	\ge
	\binom{p+q+\left\lfloor q/p\right\rfloor}{p}.
	\]
\end{cor}


This gives a lower bound for the maximum size of weakly cross-intersecting set-pair systems
with \(|A_i|=p\) and \(|B_i|=q\), and it is related to Engel's result~\cite[Theorem 6(b)]{Eng96}.

\subsection{Further work}
Engel~\cite{Eng96} singled out the three-level case \(u-\ell=2\).
Here each maximal interval 
$
[A,B]
$
consists of one \(\ell\)-set $A$, two \((\ell+1)\)-sets, and one \((\ell+2)\)-set $B$. 
Therefore
\[
\nu_{n;\ell,\ell+2}
\le
\min\left\{
\binom n\ell,\,
\left\lfloor\frac12\binom n{\ell+1}\right\rfloor,\,
\binom n{\ell+2}
\right\}.
\]
Theorem~\ref{thm:main} shows that
\[
\nu_{n;\ell,\ell+2}=\binom n\ell
\qquad\text{when}\quad
n\ge 3\ell+2.
\]
By taking complements and replacing $\ell$ by $n-\ell-2$, it also gives the opposite range
\[
\nu_{n;\ell,\ell+2}=\binom n{\ell+2}
\qquad\text{when}\quad
2n\le 3\ell+4.
\]
Thus the remaining difficulty lies in the central range.


\begin{prob}
Determine \(\nu_{n;\ell,\ell+2}\) in the remaining central range $(3\ell+4)/2< n< 3\ell+2$. In particular, is it always true that
\[
\nu_{n;\ell,\ell+2}
=
\min\left\{
\binom n\ell,\,
\left\lfloor\frac12\binom n{\ell+1}\right\rfloor,\,
\binom n{\ell+2}
\right\}?
\]
\end{prob}

The present paper focuses on the Boolean lattice.
It would be interesting to study analogous problems for consecutive ranks of other classical ranked posets, 
such as subspace lattices, divisor lattices, and partition lattices.


\newpage
\section{Appendix}
The proof of Theorem~\ref{thm:main} gives the following recursive construction. 
For each \(A\in\binom{[n]}{\ell}\), the algorithm outputs an \((\ell+r)\)-set \(B\supseteq A\), and the intervals
$
[A,B]
$
are precisely those constructed in the proof.

\begin{algorithm}[H]
	\caption{Construction of \(B\)}
	\label{alg:construction}
	\begin{algorithmic}[1]
		\Require Integers \(n\ge 1, \ell,r\ge 0\) with \(n\ge (\ell+1)\,r+\ell\), and an \(\ell\)-set \(A\subseteq[n]\).
		\Ensure An \((\ell+r)\)-set \(B\supseteq A\).
		
		\If{\(r=0\)}
		\State \Return \(A\)
		\EndIf

		\If{\(\ell=0\)}
		\State \Return \(\{1,2,\ldots,r\}\)
		\EndIf
		
		\If{\(n=(\ell+1)\,r+\ell\)}
		\State Endow the ground set \([n]=\{1,2,\ldots,n\}\) with the cyclic order.
		\For{\(1\le i\le n\)}
		\State Define \(w(i)\) by
		\Statex
		\[
		w(i)=
		\begin{cases}
			-r, & \qquad i\in A,\\
			1, & \qquad i\notin A.
		\end{cases}
		\]
		\EndFor
		\State Set \(s(0)=0\) and \(s(i)=\sum_{j=1}^{i}w(j)\) for \(1\le i\le n\).
		\State Let \(M=\max_{0\le i\le n}s(i)\).
		\For{\(m=M-r+1,M-r+2,\ldots,M\)}
		\State Let \(i_m\) be the first index such that \(s(i_m)=m\).
		\EndFor
		\State Set
		\Statex
		\[
		C_A=\{i_m: M-r+1\le m\le M\}.
		\]
		\State \Return \(A\cup C_A\)
		\EndIf
		
		\If{\(n>(\ell+1)\,r+\ell\)}
		\If{\(n\notin A\)}
		\State Regard \(A\) as an \(\ell\)-set in \([n-1]\).
		\State \Return the output of Algorithm~\ref{alg:construction} with input \((n-1,\ell,r,A)\).
		\Else
		\State Write \(A=\{n\}\cup A'\), where \(A'\in\binom{[n-1]}{\ell-1}\).
		\State Let \(B'\) be the output of Algorithm~\ref{alg:construction} with input \((n-1,\ell-1,r,A')\).
		\State \Return \(\{n\}\cup B'\)
		\EndIf
		\EndIf
	\end{algorithmic}
\end{algorithm}
\end{document}